\documentclass[12pt, a4paper]{article}

\usepackage{amsmath}\multlinegap=0pt
\usepackage[latin1]{inputenc}
\usepackage{amsthm}
\usepackage{amssymb}
\usepackage[english]{babel}
\usepackage{graphicx}
\usepackage{marginnote}
\usepackage{soul}
\usepackage{hyperref}
\hypersetup{
	colorlinks=true,
	linkcolor=blue,
	filecolor=magenta,      
	urlcolor=cyan,
	citecolor=cyan,
}

\numberwithin{equation}{section}

\theoremstyle{plain}
\newtheorem{thm}{Theorem}[section]

\newtheorem{prop}[thm]{Proposition}
\newtheorem{lem}[thm]{Lemma}
\newtheorem{defi}[thm]{Definition}

\theoremstyle{definition}

\theoremstyle{remark}
\newtheorem{rem}[thm]{Remark}

\newcommand{\R}{\mathbb{R}}

\newcommand\pref[1]{(\ref{#1})}

\let \eps\varepsilon

\newcommand\FF{{\cal F}}

\newcommand\dist{\mathop{\mathrm{dist}}\nolimits}

\newcommand\Lip{\mathrm{Lip}}

\def\<#1,#2>{\left<#1,#2\right>}

\newcommand\indi{{\mathbf{1}}}

\newcommand\tilu{\widetilde{u}}

\newcommand\tilp{\widetilde{p}}
\newcommand\tilv{\widetilde{v}}
\newcommand\tilq{\widetilde{q}}
\newcommand\tilz{\widetilde{z}}

\newcommand\BBB{{\cal {B}}}
\newcommand\AAA{{\cal {A}}}
\newcommand\tA{\widetilde{\AAA}}

\newcommand\zb{{\overline{z}}}

\newcommand\uy{{\underline{y}}}

\title{Existence of solutions to principal-agent problems with adverse selection under minimal assumptions}

\author{Guillaume Carlier\thanks{CEREMADE, UMR CNRS 7534, PSL, Universit\'e Paris IX
Dauphine, Pl. de Lattre de Tassigny, 75775 Paris Cedex 16, FRANCE and INRIA-Paris, MOKAPLAN,
\texttt{carlier@ceremade.dauphine.fr}}
\and
Kelvin Shuangjian Zhang \thanks{CNRS and D\'epartement de Math\'ematiques et Applications, \'Ecole Normale Sup\'erieure / Universit\'e PSL, Paris, FRANCE.
\texttt{szhang@ens.fr}}
}

\begin{document}

\maketitle

\begin{abstract}
We prove an existence result for the principal-agent problem with adverse selection  under general assumptions on preferences  and  allocation spaces. Instead of assuming that the allocation space is finite-dimensional or compact, we consider a more general coercivity condition which takes into account the principal's cost and the agents' preferences. Our existence proof  is simple and flexible enough to adapt to partial participation models as well as to the case of type-dependent budget constraints.

\end{abstract}


\section{Introduction}\label{sec-intro}

The principal-agent problem with adverse selection plays a distinguished role in modern microeconomic theory and has attracted a considerable amount of attention due to its numerous economic applications such as nonlinear pricing (Mussa and Rosen \cite{MussaRosen},  Armstrong \cite{Armstrong}, Rochet and Chon\'e \cite{Rochet}), taxation theory (Mirrlees \cite{Mirrlees}),  regulation (Laffont and Tirole \cite{LaffontTirole}),  to name just a few.  In these problems,  the principal cannot observe  agents' types; hence her profit maximization program is constrained by  incentive compatibility. This leads to variational problems subject to global constraints which are difficult to solve in general. The goal of the present paper is to present a rather elementary approach to the existence of optimal contracts in adverse selection problems under minimal assumptions.

\smallskip

 To fix ideas, let us consider a standard monopoly optimal nonlinear pricing problem. Denoting by $Q\subset \R_+^m$ the set of products  that are technically feasible for the monopolist and by $I$ a certain range  of prices, agents'  preferences are given by the utility $U(x,q, p)$ which depends on their (unobservable) type $x\in X$, the product  $q\in Q$ and the price $p\in I$. The monopolist knows the distribution of types $\mu$ and has a cost function denoted by $q\mapsto c(q)$. Her problem then consists of maximizing her total profit
\begin{equation}\label{defopi}
\pi:=\int_X (p(x)-c(q(x)) \mbox{d} \mu(x)
\end{equation}
among contracts $x\mapsto (q(x), p(x))\in Q\times I$,  which are incentive compatible, i.e.,
\[U(x, q(x),p(x)) \geq   U(x, q(x'),p(x')), \; \forall (x,x')\in X^2,\]
and individually rational, i.e., satisfying the participation constraint
\[ U(x, q(x),p(x)) \geq U(x, q_0, p_0), \; \forall x\in X\]
where $(q_0, p_0)$ is a certain outside option contract available to the agents. 
The multidimensional case $m>1$ is considerably harder than the unidimensional case. Indeed, when $m=1$, the standard  single-crossing condition enables one to use specific arguments based either on optimal control (as in Laffont and Guesnerie \cite{GuesnerieLaffont}) or monotonicity (as in Mussa and Rosen \cite{MussaRosen} or Jullien \cite{Jullien}). 

\smallskip

In higher dimensions, if $Q$ and $I$ are compact, the existence of an optimal contract is well-known under general assumptions on the preferences and the cost. It follows for instance from the general results of Monteiro and Page \cite{MonteiroPage} (also see Carlier \cite{Carlier} for the quasilinear case) or  Ghisi and Gobbino \cite{GhisiGobbino} who developed an elegant and original Gamma-convergence approach. More recently,  N\"oldeke and Samuelson \cite{NoldekeSamuelson15p}, McCann and Zhang \cite{McCannZhang}, and Zhang \cite{Zhang} have established general  existence results  emphasizing the role of duality and generalized convexity. N\"oldeke and Samuelson provide a general existence result assuming that the type and product spaces are compact, by  a duality argument based on Galois connections. McCann and Zhang not only show a general existence result assuming a single-crossing type condition and boundedness of the agent-type and product-type spaces, but also generalize uniqueness and convexity results of Figalli, Kim and McCann \cite{FigalliKimMcCann} to the non-quasilinear case. In the vein of Carlier \cite{Carlier}, Zhang \cite{Zhang} proves a general existence result using  generalized convex analysis,  under weaker assumptions on the product domain and without assuming the generalized single-crossing condition from \cite{McAfeeMcMillan88}.

\smallskip

Why should we bother with yet another existence result, then? Firstly,  compactness of $Q$ and/or  $I$ is a  severe restriction which rules out many important examples. In particular, upper bounds on prices should come as a result of the  model rather than as an assumption. Secondly,  for an optimization problem to have solutions, compactness of the admissible set can be replaced by a weaker assumption, which takes advantage of properties of both the objective and the constraints. For instance, in  \pref{defopi}, $\pi$ cannot be too negative and $U$ cannot be too small because of the participation constraint. This  gives extra restrictions on contracts; our assumption is that these restrictions are enough to force approximate optimizers to remain in a compact set.

\smallskip

Consider the monopoly pricing problem above in the extreme case where $X=\{x\}$ is a singleton (so that the adverse selection problem disappears) and $U(x,q, p)=b(x)\cdot q-p$. Then the optimal contract simply corresponds to setting $p=p_0 +b(x)\cdot (q-q_0)$ and finding $q$ by maximizing $b(x) \cdot q- c(q)$. 
The existence of such an optimal $q$ is obvious if $Q$ is compact and $c$ is lower semicontinuous. However, compactness  of $Q$  can be replaced by an assumption, called coercivity, which requires compactness of the smaller set $\{q\in Q \: : \; c(q) \leq c(q_0)+ b(x) \cdot (q-q_0)\}$ (which is automatically bounded if $c$ is superlinear for instance). This elementary example also shows that coercivity is indeed a minimal assumption\footnote{For instance,  if $Q=\R_+^m$, $c(q)=\sqrt{\vert q\vert}$ and $b(x)\in \R_{++}^m$ the principal's profit is unbounded from above.}. This also strongly suggests that the natural condition for the existence of optimal contracts is the relative compactness of the set of $(q,p)$ for which $p-c(q)\geq p_0-c(q_0)$ and $U(x, q,p) \geq U(x, p_0, q_0)$ for at least one type $x$, rather than the compactness of $Q$ and $I$. This is precisely the coercivity condition that we will consider (see \pref{defofK}-\pref{hyph4}) and under which we will prove existence of an optimal contract. Another (more technical at first glance) advantage of our approach is that, contrary to the references listed above, the contracts  we consider belong to a general Polish space, which can be infinite-dimensional (functions of time or random variables for instance).

\smallskip

Our  proof involves two steps. The first step consists of showing that any feasible contract can be improved by another one which yields a larger benefit than the outside option to the principal for every type of agents. This key observation is not new, it appears in Monteiro  and Page \cite{MonteiroPage} and  Carlier \cite{Carlier},  but these authors did not take advantage of it to get rid of their compactness assumptions.  Thanks to our coercivity assumption, the improved contracts remain in a compact set and the existence proof can be carried out along the same lines as, for instance,  Monteiro and Page \cite{MonteiroPage}. We then extend our existence result in two directions. The first extension considers the case where the outside option is not necessarily feasible for the principle which leads to  partial participation models  as in Jullien \cite{Jullien}. Our  second extension concerns agents facing a type-dependent budget constraint as in the works of Monteiro and Page \cite{MonteiroPage} and Che and Gale \cite{CheGale}. Type-dependent budget constraints introduce possible discontinuities in preferences. While Monteiro and Page showed that the resulting difficulties may be overcome by a certain nonessentiality assumption, we will follow a slightly different route showing that a non-atomicity condition can be used instead.

\smallskip

The paper is organized as follows. Section \ref{sec-state} presents the basic model and main assumptions. In section \ref{sec-existence}, we establish an existence result for the basic model.  Section \ref{sec-pppap} shows how to extend the existence proof to models with partial participation, as in Jullien \cite{Jullien}. In section \ref{sec-budget}, we generalize the analysis to the case of type-dependent budget constraints for the agents, as in Monteiro and Page \cite{MonteiroPage} and  Che and Gale \cite{CheGale}. Finally, we have gathered in the appendix several simple measurable selection results used throughout the paper.

\section{Problem statement and assumptions}\label{sec-state}

The agents' type space is a general probability space $(X, \FF, \mu)$.  The allocation space is denoted by $Z$ and  assumed to be a Polish space (i.e.,  a separable and completely metrizable topological space). The agents' preferences are given by a function $U$: $X\times Z \to \R$ for which we assume that
\begin{equation}\label{hyph1}
\forall x\in X, U(x, .) \mbox{ is continuous on $Z$}, 
\end{equation}
and
\begin{equation}\label{hyph2}
\forall z\in Z, \; U(., z) \mbox{ is $\FF$-measurable on $X$.}
\end{equation}

Agents have access to an outside option $z_0\in Z$. A contract is a measurable map $z$ : $X\to Z$, and it is called \emph{feasible} if it is both individually rational, i.e.,
\begin{equation}\label{ir}
U(x,z(x)) \geq U(x, z_0), \; \forall x\in X, 
\end{equation}
 and incentive compatible, i.e.,
 \begin{equation}\label{ic}
 U(x,z(x))\geq U(x, z(x')), \; \forall (x, x')\in X\times X.
  \end{equation}
 Finally  a cost function $C$ : $Z\to \R\cup\{+\infty\}$ is given for the principal which we assume to satisfy
 \begin{equation}\label{hyph3}
 C \mbox{ is  lower semicontinuous  and }  C(z_0)<+\infty.
  \end{equation}
The principal's problem is to find a cost minimizing feasible contract\footnote{From now on, we adopt the  convention that the principal is a cost minimizer instead of a profit maximizer, hoping this will not create any confusion for the reader.}:
\begin{equation}\label{pap}
\inf \Big\{ \int_X C(z(x)) \mbox{d} \mu(x) \; : z \mbox{ : $X \to Z$ feasible }\Big\}.
\end{equation}

We will prove  in the next section that \pref{pap} admits a solution under an additional coercivity assumption. Defining 
\begin{equation}\label{defofK}
K:=\overline{ \{z\in Z \; : C(z)\leq C(z_0), \; \mbox{ and } \exists x\in X \mbox{ : } U(x,z)\geq U(x,z_0)\}},
\end{equation}
our coercivity assumption is that
\begin{equation}\label{hyph4} 
K  \mbox{ is compact}.
\end{equation}
Our coercivity condition \pref{hyph4} requires  allocations which are (i) less costly than the outside option for the principal and  (ii) preferred to it by at least one type of agents, to form a relatively compact set. It is not only weaker than the compactness of $Z$ but also more natural in the sense that it takes into account both the cost and the agents' preferences. As explained in the introduction, it is not difficult to see (even when $X$ is  a singleton) that this assumption cannot be weakened if one wants  \pref{pap} to admit solutions.

\section{Existence of an optimal contract}\label{sec-existence}

\subsection{An a priori estimate}

The main argument for the existence of a solution is based on the following result:  the principal can always improve her payoff using contracts with values in $K$. This argument is not new:  a similar observation was made in \cite{MonteiroPage} and \cite{Carlier} but it was not exploited to derive existence results when $Z$ is not compact.

\begin{prop}\label{apriori}
Assume \pref{hyph1}-\pref{hyph2}-\pref{hyph3} and \pref{hyph4}. Let $z$ be a feasible contract.  Then there exists another feasible contract $\tilz$ such that $\tilz(X)\subset K$ and 
\[\int_X C(\tilz(x)) \mbox{d} \mu(x) \leq \int_X C(z(x)) \mbox{d} \mu(x).\] 

\end{prop}

\begin{proof}
We may of course assume that 
\begin{equation}\label{nonemptypp}
\{x\in X \; : \; C(z(x)) \leq C(z_0)\}\neq \emptyset
\end{equation}
since otherwise the constant contract $\tilz\equiv z_0$ satisfies the desired claim.

\smallskip

Let us assume \pref{nonemptypp} and define for every $x\in X$,
\[u(x):=U(x,z(x)).\]
By individual rationality and   incentive compatibility, one can write
\[u(x)=\max_{z'\in \AAA} U(x,z') \mbox{ where } \AAA:=\{z_0\} \cup \overline{\{z(x'), \; x'\in X\}}.\] 
Let us note that $\AAA \cap K\neq \emptyset$  and define
\[\tilu(x):=\max_{z'\in \AAA  \cap K} U(x,z').\]
We thus have $U(x, z_0) \leq \tilu(x) \leq u(x)$ and  $\tilu(x)=U(x,z(x))=u(x)$ whenever $C(z(x))\leq C(z_0)$. Since $\AAA \cap K$ is compact, the set
\[\Gamma(x):=\{ z \in \AAA \cap K \mbox{ : }   \tilu(x)=U(x,z)\}\]
 is nonempty and closed, for every $x\in X$, thanks to assumption \pref{hyph1}. Moreover, thanks to \pref{hyph2}, the set valued map $\Gamma$ has an $\FF$-measurable selection (see the Appendix for details) which we denote by $\tilz$. Note that if $C(z(x))\leq C(z_0)$ then $z(x)\in \Gamma(x)$.  We may therefore also assume that $z(x)=\tilz(x)$ for every $x\in X$ for which $C(z(x))\leq C(z_0)$. By construction, $\tilz$ is individually rational. For every $(x,x')\in X\times X$, since $\tilz(x')\in \AAA \cap K$, we have $\tilu(x)=U(x,\tilz(x))\geq U(x, \tilz(x'))$ so that $\tilz$ is also incentive compatible. Finally, $C(\tilz(x))=C(z(x))$ when $C(z(x))\leq C(z_0)$, and $C(\tilz(x))\leq C(z_0) \leq C(z(x))$ otherwise which shows that the feasible contract $\tilz$ has lower cost than the original one $z$ and it takes by construction its values in $K$.

\end{proof}

\begin{rem}\label{rmk1_1}
The economic intuition behind the proof of Proposition \ref{apriori} is quite clear: the principal is better off by removing bad contracts (i.e. contracts which are more costly than the oustide option). This argument  seemingly relies on the fact that the cost function does not depend on agents' types. However, it may be natural to allow for a cost $c(x,z)$ which is  also type-dependent (e.g. in common value problems).  For instance, if $c(x,z)=F(x, C(z))$ with $F$ increasing in its second argument, the proof above still works. Indeed, the contract $\tilz$ constructed above actually satisfies $C(\tilz(x)) \leq C(z(x))$ for every $x$: it is therefore again an improvement for the principal. More general type-dependent costs might be considered as well.  Assume that $c(x,.)$ is lower semicontinuous, that $c(x, z_0)<+\infty$ for every $x$ and that $c$ has the property that whenever $c(x,z) \leq c(x, z_0)$    for some $(x,z)\in X\times Z$ then $c(x',z) \leq c(x', z_0)$ for every $x'\in X$.  Modifying the set  $K$ defined  in  \pref{defofK} as
\begin{equation*}
K:=\overline{ \{z\in Z \;  : \; \exists x\in X \mbox{ s.t. } \;  c(x,z)\leq c(x,z_0) \; \mbox{ and }  U(x,z)\geq U(x,z_0)\}},
\end{equation*}
then Proposition \ref{apriori} still holds while replacing the integrals of $C$ in the statement by those of $c$ (and the set defined in \pref{nonemptypp} by $\{x\in X: c(x, z(x)) \le c(x, z_0)\}$). For such costs, it is not difficult to extend the existence analysis of paragraph \ref{par-exres}.

\end{rem}

\subsection{An existence result}\label{par-exres}

Proposition \ref{apriori}  enables us to reduce the principal's problem to the compact allocation space $K$ (given by \pref{defofK})  instead of $Z$. From this reduction, classical arguments along the lines of \cite{MonteiroPage}, \cite{Carlier}, \cite{Zhang} give the existence of an optimal contract:

\begin{thm}\label{existthm1}
Under assumptions \pref{hyph1}-\pref{hyph2}-\pref{hyph3} and \pref{hyph4}, the principal's problem \pref{pap} admits at least one solution.

\end{thm}

\begin{proof}
Let $(z_n)_n$ be a minimizing sequence for \pref{pap}, i.e., a sequence of feasible contracts such that
\begin{equation}\label{miniseq}
 \lim_n  \int_X C(z_n(x)) d\mu(x) =  \inf \pref{pap}.
\end{equation}
Using Proposition \ref{apriori}, we may further assume that $z_n(X)\subset K$ for each $n$. Then define
\[u_n(x):=U(x, z_n(x))=\max_{z \in \AAA_n} U(x,z) \mbox{ where } \AAA_n:=\{z_0\} \cup \overline{\{z_n(x), \; x\in X\}}.\]
Since the nonempty compact set $\AAA_n$ is included in $K$ for every $n$, we may assume, taking a subsequence if necessary, that $\AAA_n$ converges to some nonempty compact subset $\AAA^*$ of $K$ in the Hausdorff distance\footnote{Denoting by $d$  a distance that completely metrizes the topology of the separable space $Z$, and by   $\dist(A, z):=\inf_{z'\in A} d(z',z)$ the distance from $z$ to the set $A$, the Hausdorff distance between the sets $A$ and $B$ is  $d_H(A, B):=\max (\sup_{b\in B} \dist(A, b), \sup_{a\in A} \dist(B, a))$.}, i.e.,
\begin{equation}\label{hausd}
 \lim_n d_H(\AAA_n, \AAA^*) = 0.
\end{equation}
 Then define
 \begin{equation}
 u^*(x):=\sup_{z\in \AAA^*} U(x,z).
 \end{equation}
 Define also the set-valued map $x\in X\mapsto \Gamma^*(x)$ by
 \begin{equation}\label{gammas}
 \Gamma^*(x):=\{z\in  K\; : \; \exists n_j \to \infty \; \mbox{ s.t. }  \; z_{n_j}(x) \to z, \; C(z_{n_j}(x)) \to \liminf_n C(z_n(x))\}.
 \end{equation}
$\Gamma^*(x)$ is a nonempty compact set for every $x$ and our assumptions guarantee that $\Gamma^*$ has an $\FF$-measurable  selection which we denote by $z^*$ (see Lemma \ref{ms1} in the Appendix for details). Let $x\in X$ and $z$ be a cluster point of $z_n(x)$ such that $\limsup_n u_n(x)=\limsup_n U(x, z_n(x))=U(x,z)$. It follows from \pref{hausd}  that $z\in \AAA^*$, hence,
\begin{equation}\label{limsup}
\limsup_n u_n(x) \leq u^*(x).
\end{equation}
Now let $z\in \AAA^*$. Again by \pref{hausd}, there exists a sequence $(z'_n)_n$ converging to $z$ such that $z'_n \in \AAA_n$ for each $n$. By incentive compatibility, we have $u_n(x)=U(x, z_n(x))\geq U(x,z'_n)$ for each $n$ so that $\liminf_n u_n(x) \geq U(x, z)$. Taking the supremum in $z\in \AAA^*$, we get
\begin{equation}\label{liminf}
\liminf_n u_n(x) \geq u^*(x).
\end{equation}
From \pref{limsup}-\pref{liminf}, we deduce that $u_n(x)=U(x, z_n(x))$ converges to $u^*(x)$ for each $x\in X$. Choosing a subsequence of $z_n(x)$ that converges  to $z^*(x)$ therefore gives $u^*(x)=U(x, z^*(x))$.  Then, for any $(x',x)\in X\times X$, since $z^*(x')\in \AAA^*$, we have $u^*(x)=U(x, z^*(x)) \geq U(x, z^*(x'))$ which shows that $z^*$ is incentive compatible. Since $z_0\in \AAA^*$, we have $U(x, z^*(x)) \geq U(x, z_0)$  for each $x\in X$ so that $z^*$ is individually rational. Finally, Fatou's lemma (note that $C(z_n)$ is bounded from below by the minimum of $C$ on the compact set $K$) and the fact that $z^*(x)\in \Gamma^*(x)$,  where $\Gamma^*(x)$ is given by \pref{gammas}, give
\[  \inf \pref{pap} \geq   \int_X   \liminf_n C(z_n(x)) d\mu(x) \geq   \int_X   C(z^*(x)) d\mu(x),\]
so that $z^*$ solves \pref{pap}. 
 
\end{proof}

\subsection{Examples}\label{subsec-ex}

\subsubsection*{Finite-dimensional allocations: quasilinear preferences}

Consider the simple example in which $z=(p,q)\in \R\times \R^d$ where $p\in \R$ represents the price of the contract and $q\in \R^d$ is a list of product attributes. Assume that preferences are quasi linear, i.e.,
\[U(x,z) =b(x,q)-p,\]
where $b(., q)$ is measurable and $b(x,.)$ is Lipschitz with the same  Lipschitz constant $\Lip_{b}$ for every $x\in X$.  Let us also assume separability of the cost:
\[C(z)=c(q)-p,\]
with $c$ : $\R^d \to \R\cup\{+\infty\}$ lower semicontinuous and superlinear, i.e., such that 
\[\lim_{\Vert q \Vert \to \infty} \frac{c(q)}{\Vert q \Vert} =+\infty.\]

Denoting by $(p_0,q_0)$ the outside option, assume that $c(q_0)<+\infty$. If $c(q)-p\leq c(q_0)-p_0$ and $b(x,q)-p\geq b(x,q_0)-p_0$ for some $x\in X$, since $b(x,.)$ is $\Lip_{b}$ Lipschitz, we get  
\[c(q)\leq c(q_0) +p-p_0 \leq c(q_0)+ \Lip_{b} \Vert q- q_0\Vert\]
and the fact that $c$ is superlinear gives a bound on $q$. The bound on $p$ then directly follows. This shows that the closed set $K$ defined by  \pref{defofK} is bounded hence compact.

\subsubsection*{Finite-dimensional allocations: fully nonlinear preferences}

Following McCann and Zhang \cite{McCannZhang}, consider now a general nonlinear utility function of the form
\begin{flalign*}
	U(x,z) = G(x, q, p),
\end{flalign*}
where $z = (p, q) \in \R \times \R^d$ represents a contract and $x$ a type. Assume that $G$ is strictly decreasing in the price $p$, which means  that the same product with a higher price provides less utility to agents.

Each contract $z$ has a cost for the principal, which is
\begin{flalign*}
	C(z) = c(q) -p.
\end{flalign*}

Assume $c$ is superlinear in $q$, $G$ satisfies $\partial_{p} G(x, q, p) \le -\lambda <0$ for all $(x, z) \in X \times Z$ and $G(x, \cdot, p_0)$ is Lipschitz, uniformly in $x$, with  Lipschitz constant $\Lip_{G}$.  To show that $K$ is bounded, it is useful to define $K_1 = K \cap \{(q, p) \in  \R^d \times \R : p \le p_0 \}$ and $K_2 = K\setminus K_1$. 

By definition, we know that for any $(q, p) \in K_1$,
\begin{flalign*}
	c(q) - p_0 \le c(q) - p \le c(q_0) - p_0.
\end{flalign*} 
Since $c$ is superlinear, this implies  that $q$ is bounded. Since $c(q) -c(q_0) +p_0 \le p \le p_0$, $p$ is also bounded. Thus, $K_1$ is bounded.

\smallskip

 Now, if $(p,q)\in K_2$  there exists $x\in X$ such that $G(x, q, p)  \ge G(x, q_0, p_0)$, but since $p>p_0$ and  $\partial_{p} G \le -\lambda$, using the Lipschitz assumption on  $G(x,., p_0)$, we have
\[\begin{split}
G(x,q_0,p_0)\le G(x,q,p) \le G(x, q,p_0)  -\lambda(p-p_0)  \\ 
\le G(x,q_0,p_0)  -\lambda(p-p_0) + \Lip_{G} \Vert q - q_0 \Vert; 
\end{split}\]
hence 
\[0\leq p-p_0 \leq  \frac{  \Lip_{G}}{\lambda}  \Vert q - q_0 \Vert, \; c(q) \le c(q_0)+  \frac{  \Lip_{G}}{\lambda}  \Vert q - q_0 \Vert. \]

Since $c$ is superlinear, this implies $q$ is bounded, so is $p$ as well, and therefore $K_2$ is bounded. This shows that $K$ is compact.

\subsubsection*{Infinite-dimensional allocations}

We now consider the possibility that the allocation $z$ is infinite-dimensional (see \cite{BalderYannelis} for contracts taking values in a space of random variables), one can think for instance of a time-dependent function. We consider contracts of the form $z=(p,q)$ with $p\in\R$ and  $q\in Z:=L^2((0,T), \R^d)$, a utility of the form
\[U(x,z):=\int_0^T v(t,x,q(t))\mbox{d}t-p,\]
a cost 
\[C(z)=\int_0^T (c(t,q(t))+ \vert \dot{q}(t)\vert^2)  \mbox{d}t  -p,\]
(with $C=+\infty$ whenever $\dot{q}$ is not $L^2$) and an outside option $(z_0, q_0)$ with $\dot{q_0} \in L^2$. Then if $c(t,.)$ is superlinear uniformly in $t$ and $v(t,x,.)$ is Lipschitz uniformly in $(t,x)$, the set $K$ consists of $(p, q)\in \R\times L^2$ such that both $\int_0^T (\vert q \vert + \vert \dot{q}\vert^2) \mbox{d}t$ and $p$ are uniformly bounded; it is therefore compact in $\R\times L^2$ by the Rellich-Kondrachov Theorem (see \cite{brezis}).

\section{Partial participation}\label{sec-pppap}

In the model of section \ref{sec-state}, we assumed that the outside option $z_0$ belongs to the set of feasible contracts for the principal and has a finite cost. We also imposed the participation constraint for all agents, excluding the possibility of partial participation. If $C(z_0)\leq 0$, there is no real loss of generality in imposing full participation but if $C(z_0)>0$, the principal may have an interest in excluding some agents.  This more delicate situation was analyzed by Jullien \cite{Jullien} (in an otherwise standard quasilinear unidimensional framework). Our aim is to show that our approach to existence of an optimal contract can be extended to the partial participation case.

\subsection{Model and assumptions}

We assume that the agents' preferences are as in section \ref{sec-state}, i.e., they satisfy \pref{hyph1}-\pref{hyph2}. We are also given a type dependent reservation utility $u_0$: $X\to \R$ which  is assumed to be $\FF$-measurable. The principal's cost function $C$ : $Z\to \R\cup\{+\infty\}$ is lower semicontinuous. Given an incentive compatible contract $x\in X\mapsto z(x)$, we denote by $p_z$ the participation set:
\[p_z:=\{x\in X \; : \; U(x, z(x)) \geq u_0(x)\}\]
and we denote by $\indi_{p_z}$ its characteristic function:
\[\indi_{p_z}(x)=\begin{cases} 1 \mbox{ if  $U(x, z(x)) \geq u_0(x)$} \\ 0 \mbox{ otherwise.}\end{cases}\]
 The main departure from the full participation model of  section \ref{sec-state} is that instead of imposing the participation constraint, the principal's total cost will be integrated only on the participation set. We will assume that the principal can make nonnegative profit; that is,
 \begin{equation}\label{hyppp1}
\exists \zb\in Z \: : \;  C(\zb)\leq 0 \mbox{ and } \{x\in X \; : \; U(x, \zb) \geq u_0(x)\} \neq \emptyset
  \end{equation}
and that the set  
 \begin{equation*}
 F_0:=\overline{\{z \in Z \; : \; C(z) \leq 0 \mbox{ and } \exists x\in X \; : U(x, z) \geq u_0(x)\}},
 \end{equation*}
 which is nonempty thanks to \pref{hyppp1}, satisfies
  \begin{equation}\label{hyppp2}
F_0  \mbox{ is compact}.
  \end{equation}

The principal's problem then reads
\begin{equation}\label{pappartial}
\inf \Big\{ \int_X \indi_{p_z}(x) C(z(x)) \mbox{d} \mu(x) \; : z \mbox{ : $X \to Z$ incentive compatible}\Big\}.
\end{equation}

\subsection{Existence of an optimal contract}

\begin{prop}\label{aprioripartial}
Assume  that $C$ : $Z \to \R\cup\{+\infty\}$ is lower semicontinuous, \pref{hyph1}-\pref{hyph2}-\pref{hyppp1} and \pref{hyppp2}. Let $z$ be an incentive compatible  contract.  Then there exists another  incentive compatible contract $\tilz$ such that $\tilz(X)\subset F_0$ and 
\begin{equation}\label{improvpp}
 \indi_{p_{\tilz}}(x) C(\tilz(x)) \leq \indi_{p_z}(x) C(z(x)), \; \forall x\in X.
 \end{equation} 

\end{prop}

\begin{proof}
If $z(X) \subset Z\setminus F_0$ then $\indi_{p_z}(x) C(z(x))\geq 0$ for every $x\in X$; hence, thanks to \pref{hyppp1}, the constant contract $\tilz\equiv \zb$ has the desired properties. We thus now assume that $z(X)\cap F_0 \neq \emptyset$ and argue in a similar way as in the proof of Proposition \ref{apriori}. Define
\[u(x):=U(x, z(x))=  \max\{ U(x, z') \; : \;   z'\in \overline{z(X)} \} \]
and 
\[\tilu(x):=  \max \{ U(x, z') \;  : \;   z'\in \overline{z(X)}\cap F_0\}\]
and let $x\mapsto \tilz(x)$ be a measurable map such that $\tilz(x) \in  \overline{z(X)}\cap F_0$ and $\tilu(x)=U(x, \tilz(x))$ for every $x\in X$. Note that if $x\in p_z$ and $C(z(x)) \leq 0$ then $z(x)\in F_0$. We can therefore also impose $\tilz(x)=z(x)$ and $u(x)=\tilu(x)$ whenever $x\in p_z$ and  $C(z(x)) \leq 0$. By construction, $\tilz$ is incentive compatible and takes values in $F_0$; in particular $C(\tilz(x)) \leq 0$ for every $x\in X$. If $x\notin p_z$, or, if $x\in p_z$ and $C(z(x))>0$, then  \pref{improvpp} is obvious. Finally, if $x\in p_z$ and $C(z(x))\leq 0$, then $z(x)=\tilz(x)$ and $u(x)=U(x, z(x))=\tilu(x)=U(x, \tilz(x))$, so that $x\in p_{\tilz}$ and  \pref{improvpp} is an equality.

\end{proof}

\begin{thm}\label{existthmpartial}
Assume that $C$ : $Z \to \R\cup\{+\infty\}$ is lower semicontinuous, \pref{hyph1}-\pref{hyph2}-\pref{hyppp1} and \pref{hyppp2}. Then the principal's problem \pref{pappartial} admits at least one solution.

\end{thm}

\begin{proof}
Thanks to Proposition \ref{aprioripartial}, we can find  a minimizing sequence $(z_n)_n$ for  \pref{pappartial} which satisfies $\AAA_n:=z_n(X) \subset F_0$. We now proceed as in the proof of Theorem \ref{existthm1}, by finding a subsequence of $\AAA_n$ which converges in the Hausdorff distance to some compact  subset of $F_0$ denoted $\AAA^*$. We then define $u^*(x):=\max_{z'\in \AAA^*} U(x, z')$ for every $x\in X$. Thanks to Lemma \ref{ms1}, there exists  a measurable map $x\in X \mapsto z^*(x)$ such that for every $x\in X$, one has  $z^*(x) \in  \Gamma^*(x)$, where
\begin{equation}\label{gammasf}
 \Gamma^*(x):=\{z\in  F_0\; : \; \exists n_j \to \infty \; \mbox{ s.t. }  \; z_{n_j}(x) \to z, \; C(z_{n_j}(x)) \to \liminf_n C(z_n(x))\}.
 \end{equation}
 Arguing as  in the proof of Theorem \ref{existthm1}, we find that $u_n:=U(., z_n(.))$ converges pointwise to $u^*$, that $u^*=U(., z^*(.))$, and  that $z^*$ is incentive compatible. To conclude that $z^*$ solves  \pref{pappartial}, thanks to Fatou's Lemma it is enough to show that 
\begin{equation}\label{wwwpp}
\indi_{p_{z^*}} (x) C(z^*(x)) \leq \liminf_n  \indi_{p_{z_n}}(x) C(z_n(x)), \; \forall x\in X.
 \end{equation}
 By construction and lower semicontinuity of $C$, we have  $C(z^*)\leq 0$ and $C(z^*) \leq   \liminf_n C(z_n)$. Let $x\in X$, and $n_j$ be  a subsequence such that both  $ \indi_{p_{z_{n_j}}}(x)$ and $C(z_{n_j}(x))$ converge. If $\indi_{p_{z_{n_j}}}(x)$ converges to $1$ then by convergence of $u_n$ to $u^*$ we have $\indi_{p_{z^*}}(x)=1$ so that \pref{wwwpp} holds. If, on the contrary,  $\indi_{p_{z_{n_j}}}(x)$ converges to $0$, \pref{wwwpp} also holds since  $\indi_{p_{z^*}}(x) C(z^*(x)) \leq 0$. 
\end{proof}

\section{The budget-constrained case}\label{sec-budget}

We now extend the model of section \ref{sec-state} and our existence result to the case where agents have a (type-dependent) budget constraint. This case is relevant in applications; it was considered by Che and Gale \cite{CheGale} and   analyzed from the existence perspective by Monteiro and Page \cite{MonteiroPage}. The authors of  \cite{MonteiroPage}  were able to deal with the discontinuity resulting from the budget constraint thanks to a specific assumption called  nonessentiality which we will not use here. Instead, we will use a non-atomicity assumption on the type distribution.

\subsection{Model and assumptions}

We consider the following setting for the budget-constrained principal-agent problem. The type of the agents will consist of a preference parameter $x$ and a budget $y$. The set  of preference parameters is denoted by $X$  which is equipped with a $\sigma$-algebra $\FF$. The set of budgets is a closed interval $Y$ with a finite lower bound $\uy$ and it is equipped with its Borel algebra which we denote by $\BBB$.  Contracts consist of pairs $(p,q)$ where $p\in \R$ denotes the price of the contract and $q$ denotes an allocation, while the set of allocations is denoted by $Q$ which is assumed to be a Polish space. The outside option $(p_0, q_0)\in \R\times Q$ is assumed to satisfy
\begin{equation}\label{cbpoorest}
p_0 \leq \uy
\end{equation}
which makes it affordable even to agents with the lowest budget. Preferences are given by  a function $V$ : $X\times \R\times Q\to \R$ and we assume that 
\begin{equation}\label{hyph1cb}
\forall x\in X, V(x, ., .) \mbox{ is continuous on $\R\times Q$}, 
\end{equation}
and
\begin{equation}\label{hyph2cb}
\forall (p,q)\in \R\times Q, \; V(., p,q) \mbox{ is $\FF$-measurable on $X$.}
\end{equation}
The joint distribution of types $(x,y)$ is given by a probability measure $\theta$ on $X\times Y$ (equipped with the product $\sigma$-algebra $\FF\otimes \BBB$). Finally the cost for the principal is given by a function $C$ : $\R\times Q\to \R\cup\{+\infty\}$ which we assume to satisfy
 \begin{equation}\label{hyph3cb}
 C \mbox{ is  lower semicontinuous and }  C(p_0, q_0)<+\infty.
  \end{equation}

 \begin{defi}
 A budget-constrained-feasible contract is an $\FF\otimes \BBB$-measurable map $(x,y)\in X\times Y\mapsto (p(x,y), q(x,y))\in \R\times Q$ that satisfies:
 
 \begin{itemize}

 \item the budget constraint: $p(x,y)\leq y$, for every $(x,y)\in X\times Y$;
 
 \item individual rationality, $V(x,p(x,y), q(x,y)) \geq V(x, p_0, q_0)$, for every $(x,y)\in X\times Y$;
 
 \item budget-constrained incentive compatibility, i.e., for every $(x,y,x',y')\in (X\times Y)^2$ if $p(x',y')\leq y$ then 
 \begin{equation}\label{cbic}
 V(x, p(x,y), q(x,y))\geq V(x, p(x',y'), q(x',y')).
 \end{equation}
 \end{itemize}

 \end{defi}

The budget-constrained principal's problem then reads 
\begin{equation}\label{papcb}
\inf \Big\{ \int_{X\times Y} C(p(x,y), q(x,y)) \mbox{d} \theta(x,y) \; : (p,q) \mbox{ budget-constrained-feasible}\Big\}.
\end{equation}

To prove that \pref{papcb} admits solutions, we shall need two more technical assumptions. The first one is a coercivity assumption similar to \pref{hyph4}.  Define $\Gamma$ as the closure of the set of $(p,q)\in \R\times Q$ such that  $C(p,q)\leq C(p_0, q_0)$, and there exists $(x,y)\in X\times Y$ such that $p\leq y$ and   $V(x,p,q)\geq V(x,p_0, q_0)$. Our coercivity assumption is then that
\begin{equation}\label{hyph4cb} 
\Gamma  \mbox{ is compact}.
\end{equation}
Our last assumption is a non-atomicity condition that will enable us to deal with the possible discontinuities caused by the budget constraint on the indirect utility function. Our non-atomicity condition is that for every measurable subset $S$ of $X\times Y$, one has\footnote{When $X$ is a Polish space, by the disintegration Theorem, $\theta$ can be disintegrated with respect to its first marginal $\alpha$ as  $\theta(\mbox{d}x, \mbox{d}y) =\theta(\mbox{d}y \vert x) \alpha(\mbox{d}x)$; in this case, condition \pref{nonatcond}  amounts to saying that for $\alpha$-a.e. $x$, the conditional probability $\theta(.\vert x)$ is atomless.} 
\begin{equation}\label{nonatcond}
\theta(S)=0 \mbox{ whenever $S_x$ is at most countable for every $x\in Q$.}
\end{equation}
Here, given $x\in X$, $S_x$  denotes the slice $S_x:=\{y\in Y \; : \; (x,y)\in S\}$.

\subsection{Existence}

Our first step in the existence proof is the following variant of Proposition \ref{apriori}:
\begin{lem}\label{aprioricb}
Assume \pref{cbpoorest}-\pref{hyph1cb}-\pref{hyph2cb}-\pref{hyph3cb} and \pref{hyph4cb}. Let $(p,q)$ be a budget-constrained-feasible contract.  Then there exists another  budget-constrained-feasible contract $(\tilp, \tilq)$ such that $(\tilp, \tilq)(X\times Y)\subset \Gamma$ and 
\[\int_{X\times Y} C(\tilp(x,y), \tilq(x,y)) \mbox{d}  \theta(x,y) \leq \int_{X\times Y} C(p(x,y), q(x,y)) \mbox{d} \theta(x,y).\] 

\end{lem}

\begin{proof}
As in the proof of Proposition \ref{apriori}, there is no loss of generality in assuming that 
\begin{equation}\label{nonemptyppcb}
\{(x,y)\in X\times Y \; : \; C(p(x,y), q(x,y)) \leq C(p_0, q_0)\}\neq \emptyset.
\end{equation}

Let us define 
\[v(x,y):=V(x, p(x,y), q(x,y)), \; \forall (x,y)\in X\times Y\]
and observe that by individual rationality and budget-constrained incentive compatibility, $v$ can be expressed as
\[v(x,y):=\max_{(p,q)\in \AAA, \; p\leq y} V(x,p,q),\]
where  
\[\AAA:=\{(p_0,q_0)\} \cup \overline{\{(p(x',y'), q(x',y')), \; (x',y')\in X\times Y\}}.\]
Since $\AAA     \cap \{(p,q)\in \R \times Q: p\le y\}   \cap \Gamma$ is non-empty and compact, thanks to \pref{hyph1cb}, we can define the following function (that is everywhere finite):
\[\tilv(x,y):=\max_{(p,q)\in \AAA\cap \Gamma, \; p\leq y} V(x,p,q).\]
Moreover, thanks to Lemma \ref{ms3}, we can choose a maximizer $(\tilp(x,y), \tilq(x,y))$ in the program above which depends in a measurable way on $(x,y)$ and we can also assume that 
\[(\tilp(x,y), \tilq(x,y))=(p(x,y),q(x,y)) \mbox{ whenever } C(p(x,y), q(x,y))\leq C(p_0, q_0).\]
Arguing as in the proof of Proposition \ref{apriori}, we deduce that $(\tilp, \tilq)$ is a budget-constrained-feasible contract and $C(\tilp(x,y), \tilq(x,y))\leq C(p(x,y), q(x,y))$ for every $(x,y)\in X\times Y$.
\end{proof}

We then have the existence result:

\begin{thm}\label{existbcopt}
Assume \pref{cbpoorest}-\pref{hyph1cb}-\pref{hyph2cb}-\pref{hyph3cb}-\pref{hyph4cb} and \pref{nonatcond}. Then \pref{papcb} admits at least one solution. 

\end{thm}

\begin{proof}
Let $(p_n, q_n)$ be a minimizing sequence for \pref{papcb}; thanks to Lemma \ref{aprioricb} there is no loss of generality in assuming that $(p_n, q_n)(X\times Y)\subset \Gamma$ where $\Gamma$ is the compact set defined above assumption \pref{hyph4cb}. Defining
\[v_n(x,y):=V(x, p_n(x,y), q_n(x,y)), \; \forall (x,y)\in X\times Y, \]
budget-constrained feasibility  then gives the representation 
\[v_n(x,y)=\max_{(p,q)\in \AAA_n, \; p\leq y}  V(x,p,q), \]
where 
\[\AAA_n:= \{(p_0,q_0)\} \cup \overline{\{(p_n(x',y'), q_n(x',y')), \; (x',y')\in X\times Y\}}.\]
Since each compact set $\AAA_n$ is included in the compact set $\Gamma$, we may also assume, passing to a subsequence if necessary, that there is a compact subset $\AAA^*$  of $\Gamma$, containing $(p_0, q_0)$ such that
\begin{equation}\label{hausdcb}
 \lim_n d_H(\AAA_n, \AAA^*) = 0.
\end{equation} 
Then define
\[v^*(x,y)=\max_{(p,q)\in \AAA^*, \; p\leq y}  V(x,p,q).\]
Thanks to Lemma \ref{ms1}, there exists a measurable selection $(p^*, q^*)$ of the set-valued map defined for every $(x,y)\in X\times Y$ by
\[\begin{split}
\Gamma^*(x,y):=\{(p,q)\in \Gamma \; : \; \exists n_j \to \infty \; \mbox{ s.t.}  \; (p_{n_j}(x,y), q_{n_j}(x,y)) \to (p,q),\\ 
 \; C(p_{n_j}(x,y), q_{n_j}(x,y)) \to \liminf_n C(p_n(x,y), q_n(x,y))\}.
 \end{split}\]
Note that by Fatou's Lemma, 
\begin{equation}\label{fatoucb}
\int_{X\times Y} C(p^*(x,y), q^*(x,y)) \mbox{d} \theta(x,y) \leq \inf \; \pref{papcb}.
\end{equation}
If $(p^*,q^*)$ was budget-constrained-feasible, the proof would be complete, but it is not necessarily the case that $(p^*, q^*)$ is budget-constrained incentive compatible (and this is where assumption  \pref{nonatcond}  comes into play). Note that by construction, using budget-constrained-feasibility of $(p_n, q_n)$, we obviously have that for every $(x,y)\in X\times Y$, $p^*(x,y)\leq y$, $V(x, p^*(x,y), q^*(x,y)) \geq V(x, p_0, q_0)$; note also that $(p^*(x,y), q^*(x,y))\in \AAA^*$ because of \pref{hausdcb}. In particular, since $p^*(x,y)\leq y$, this gives
\begin{equation}\label{inegiccb}
v^*(x,y) \geq V(x, p^*(x,y), q^*(x,y)), \; \forall (x,y)\in X\times Y.
\end{equation}
From \pref{hyph1cb} and \pref{hausdcb}, it is easy to deduce that
\begin{equation}\label{limsupvn}
v^* \geq \limsup_n v_n.
\end{equation}
Now observe that $v^*$ is nondecreasing and upper semi-continuous with respect to its second argument. Hence, defining
\[v^*_{-}(x,y):=\lim_{\eps \to 0^+} v^*(x,y-\eps), \; \forall x\in X, \forall y\in Y\setminus \{\uy\}, v^*_-(x, \uy):=v^*(x, \uy),\]
 the (measurable) \emph{singular} set
\[S:=\{(x,y)\in X\times Y \; : \; v^*(x,y)>v^*_-(x,y)\}\]
has at most countable slices $S_x$ for every $x\in X$. Assumption \pref{nonatcond} thus implies that $\theta(S)=0$. Note also that, again by  \pref{nonatcond}, $\theta(X\times \{\uy\})=0$. Therefore the  \emph{regular} set
$R:=(X\times (Y\setminus \{\uy\})) \setminus S$ is of full measure for $\theta$. Now, let $(x,y)\in R$ and $\eps>0$ be such that $y-\eps \in Y$; by compactness of $\AAA^*$  and definition of $v^*$ there is a $(p,q)\in \AAA^*$ such that $p \leq y-\eps$ and $v^*(x,y-\eps) =V(x, p,q)$. Thanks to \pref{hausdcb}, there is a sequence $(p_{n}, q_{n})$ converging to $(p,q)$  with $(p_{n}, q_{n})\in \AAA_{n}$ for every $n$  and $p_n \leq y$ for large enough $n$ so that
\[\liminf_n v_n(x,y) \geq \liminf_n V(x, p_n, q_n) =v^*(x,y-\eps).\]
Letting $\eps\to 0^+$ thus gives
\begin{equation}\label{liminfvn}
\liminf_n v_n(x,y) \geq v^*_-(x,y).
\end{equation}
Recalling \pref{limsupvn} and using the fact that $v^*=v^*_-$ on $R$, we deduce that
\begin{equation}
v_n \to v^* \mbox{ on $R$}.
\end{equation}
In particular, if $(x,y)\in R$, since $v_n(x,y)=V(x, p_n(x,y), q_n(x,y))$ converges to $v^*(x,y)$, choosing a subsequence of $(p_n(x,y), q_n(x,y))$ converging to $(p^*(x,y), q^*(x,y))$ gives
\[v^*(x,y)=V(x, p^*(x,y), q^*(x,y)).\]
This enables us to conclude that for every $(x,y)\in R$ and any $(x',y')\in X\times Y$ such that $p^*(x',y') \leq y$, one has $v^*(x,y)=V(x, p^*(x,y), q^*(x,y))\geq V(x, p^*(x',y'), q^*(x',y'))$. The last step is to modify the contract $(p^*, q^*)$ on a negligible set to make it budget-constrained feasible. To do this, first set 
\[\tA:=\{(p_0,q_0)\} \cup \overline{\{(p^*(x',y'), q^*(x',y')), \; (x',y')\in R\}}\] 
and 
\[\tilv(x,y):=\max_{(p,q)\in \tA, \; p\leq y} V(x,p,q),\]
and let $(\tilp, \tilq)$ be a measurable selection of the set-valued map $(x,y)\mapsto \{(p,q)\in \tA :  \; p\leq y, \; \tilv(x,y)=V(x,p,q)\}$. Since $\tilv(x,y)=v^*(x,y)=V(x,p^*(x,y), q^*(x,y))$ when $(x,y)\in R$, we may further impose that $(\tilp, \tilq)$ and $(p^*, q^*)$  coincide on $R$, hence $\theta$-almost everywhere. Then, it is straightforward to check that $(\tilp, \tilq)$ is budget-constrained feasible,  and the fact that it solves \pref{papcb}  directly follows from \pref{fatoucb} and the fact that $(\tilp, \tilq)=(p^*, q^*)$  $\theta$-almost everywhere. 

\end{proof}

\section*{Appendix: On measurable selections}
We have invoked the existence of measurable selections  of certain set-valued maps  several times; here, we gather some detailed justifications for the existence of such maps. Given a measurable space $(X, \FF)$, a Polish space $Z$ and a set valued map $\Gamma$ : $X \to 2^Z$ with nonempty values, a measurable selection of $\Gamma$ is by definition an $\FF$-measurable (single-valued) map $z$: $X\to Z$ such that $z(x)\in \Gamma(x)$ for all $x\in X$.  A general existence result for measurable selections is given by the Kuratowski and Ryll-Nardzewski Theorem (see \cite{KRN} and also the survey by Himmelberg \cite{Himmelberg}) which ensures that whenever
\begin{itemize}
\item  $\Gamma(x)$ is closed and nonempty for every $x\in X$, and

\item $\Gamma$ is weakly measurable in the sense that for every \emph{open} subset $U$ of $Z$, the set $\Gamma^{-1}(U):=\{x\in X \; : \; \Gamma(x) \cap U  \neq \emptyset\}$ belongs to $\FF$ 

\end{itemize}
 then $\Gamma$ admits a measurable selection.

\smallskip

In fact we do not use the full generality of the  Kuratowski and Ryll-Nardzewski Theorem.  The set-valued maps we have encountered through the paper satisfy a stronger measurability property than the one above; namely they satisfy that for every \emph{closed} subset $F$ of $Z$, $\Gamma^{-1}(F)$ belongs to $\FF$ (to see that it implies weak measurability it is enough to write the open set $U$ as a countable union of closed sets). The first measurable selection result  we have used is the following:

\begin{lem}\label{ms1}
Let $K$ be a nonempty compact subset of $Z$, $z_n$  a sequence of measurable maps, $z_n$:  $X\to K$,  and  $C$: $K\to \R\cup\{+\infty\}$ be lower semicontinuous and not identically $+\infty$ on $K$. For all $x\in X$, let 
\[\Gamma(x):=\{z\in K \; : \; \exists n_j \to \infty \; : \; z_{n_j}(x) \to z, \; C(z_{n_j}(x)) \to \liminf_n C(z_n(x))\}.\]
Then $\Gamma$ admits a measurable selection.

\end{lem}

\begin{proof}
It is easy to check that $\Gamma(x)$ is a nonempty and closed subset of $Z$ for every $x\in X$. As explained above, a sufficient condition for the existence of a measurable selection is that $\Gamma^{-1}(F)$ is measurable whenever $F$ is closed, but it is easy to check that 
\[\Gamma^{-1}(F)=\{ x\in X \; : \; \liminf_n (\dist(z_n(x), F)+ C(z_n(x)))=\liminf_n C(z_n(x))\},\]
which, written in this way, is obviously measurable.

\end{proof}

In the proof of Theorem \ref{existthm1}, we have used:

\begin{lem}\label{ms2}
Let  $U$ satisfy \pref{hyph1}-\pref{hyph2}, $A$ be a nonempty compact subset of $Z$ and set for every $x\in X$, 
\[v_A(x):=\max_{z\in A} U(x,z).\]
Then $v_A$ is measurable. Moreover, if we define for every $x\in X$, 
\[\Gamma_A(x):=\{z\in A \; : \; U(x,z)=v_A(x)\},\]
 $\Gamma_A$ admits a measurable selection.

\end{lem}

\begin{proof}
The fact that $v_A$ is measurable follows by taking $\{a_n\}_n$ a countable and dense subset of $A$ and writing $v_A(x)=\lim_n \max_{k\le n} U(x, a_k)$. Obviously, $\Gamma_{A}(x)$ is nonempty and closed for every $x\in X$. Now, if $F$ is a closed subset of $A$, we claim that $\Gamma_A^{-1}(F)$ is measurable.  As $\Gamma_A^{-1}(F)$ is empty when $A\cap F=\emptyset$, we may assume that the (compact) set $A\cap F$ is nonempty; $\Gamma_A^{-1}(F)$ then is the set where $v_A$ and $v_{A\cap F}$ coincide. It is therefore measurable.

\end{proof}

The following variant of Lemma \ref{ms2} was used for the budget-constrained model:

\begin{lem}\label{ms3}
Let $V$ satisfy \pref{hyph1cb}-\pref{hyph2cb}, $A$ be a nonempty compact subset of $\R\times Q$ and set for every $(x,y)\in X\times Y$
\[v_A(x,y):=\max_{(p,q)\in A, p\leq y} V(x,p,q),\]
(with the convention that $v_A(x,y)=-\infty$ whenever $A\cap (-\infty, y]\times Q=\emptyset$).  Then $v_A$ is measurable. Moreover, defining for every $(x,y)\in X\times Y$, 
\[\Sigma_A(x,y):=\{(p,q)\in A \; : p\le y, \; V(x,p,q)=v_A(x,y)\},\]
$\Sigma_A$ admits an $\FF\otimes \BBB$-measurable selection.

\end{lem}

\begin{proof}
For $\lambda>0$ set
\[v_A^{\lambda}(x,y):= \max_{(p,q)\in A} \{V(x,p,q)-\lambda (p-y)_+)\}.\]
Thanks to Lemma \ref{ms2}, $v_{A}^{\lambda}$ is measurable and it is easy to check  that $v_A^{\lambda}$ converges in a nonincreasing way to $v_A$ as $\lambda\to +\infty$, which shows that $v_A$ is measurable. The fact that $\Sigma_A$ admits a measurable selection can then be shown as in the proof of Lemma \ref{ms2}. 

\end{proof}

{ \bf{Acknowledgements:}} Guillaume Carlier is grateful to the Agence Nationale de la Recherche for its support through the projects MAGA (ANR-16-CE40-0014) and MFG (ANR-16-CE40-0015-01).
	Kelvin Shuangjian Zhang is thankful for the support of Mitacs Globalink Research Award during his visit to Guillaume Carlier at MOKAPLAN, INRIA-Paris in the summer 2018. Both authors are grateful to an anonymous referee who suggested the partial participation models  in section \ref{sec-pppap} and to Brendan Pass for his helpful comments.

\bibliographystyle{plain}

\bibliography{existence_minimal_assumptions}

\end{document}